\documentclass[11pt,a4paper,reqno]{amsart}
\usepackage[latin1]{inputenc}
\usepackage[english]{babel}
\usepackage{amsmath}
\usepackage{amsfonts}
\usepackage{amssymb}
\usepackage{mathrsfs}
\usepackage{latexsym}
\usepackage{yfonts}
\usepackage{natbib}

\usepackage[margin=2cm]{geometry}

\usepackage{color}

\usepackage{graphicx,color}

\usepackage[plainpages=false,colorlinks,hyperindex,bookmarksopen,linkcolor=red,citecolor=blue,urlcolor=blue]{hyperref}

\bibpunct{[}{]}{;}{n}{,}{,}

\newtheorem{tm}{Theorem}[section]

\newtheorem{defin}{Definition}

\newtheorem{remark}{Remark}[section]

\numberwithin{equation}{section}

\allowdisplaybreaks

\author{Mirko D'Ovidio}
        \address{Department of Basic and Applied Sciences for Engineering\newline
Sapienza University of Rome\newline via A. Scarpa 10, Rome, Italy}
        \email[Corresponding author]{mirko.dovidio@uniroma1.it}

\author{Paola Loreti}
        \address{Department of Basic and Applied Sciences for Engineering\newline
Sapienza University of Rome\newline via A. Scarpa 16, Rome, Italy}
        \email{paola.loreti@uniroma1.it}

\begin{document}

\title{  Solutions of   fractional logistic equations\\ by 
Euler's numbers\\ }
\maketitle

\begin{abstract} 
In this paper, we solve  in the convergence set,  the
fractional logistic
equation making use of  Euler's numbers.
To our knowledge, the answer is still an open question.
The key point is that the  coefficients  can be
connected with Euler's numbers, and then they  can be explicitly given.
The constrained of our approach  is that the formula is not valid outside
the convergence set,
 The idea of the proof consists to   explore some analogies with logistic
function and Euler's numbers,
 and then to generalize them in  the fractional case.
\end{abstract}


\section{Introduction}
\subsection{The logistic function}
A logistic function is
$$  u(t)=\frac {u_0 }   {u_0+(1-u_0)e^{-Mt} }$$
with  $M$ a positive constant and $u_0$ positive and less than 1.
The function was introduced by Pierre Francois Verhulst \cite{Ve}  to
model the  population growth.
At the beginning  of the process the growth of the population is fast;
then, as saturation process begins, the growth slows, and then growth is
close to be flat.

The logistic function is solution of the logistic differential equation 
$$u'(t)=M(u(t)-u^2(t))\,\,\,\ $$
with initial condition
$$u(0)=u_0. $$

The key assumptions in the logistic model are
 \begin{itemize}
 \item
 The population is composed by individual not distinguishable;
\item
 The population is isolated;
 \item
Self-limiting growth, that is  an intrinsic mechanism of saturation holds
  when the density of population reaches
a certain level.
 \end{itemize}
These basic assumptions may be checked in laboratory for biological
diffusion. At least for bounded time they agree to the experience, hence they may be adopted to describe phenomena  as biological models of tumour growth (\cite{tumour1, tumour2, tumour3,kirsc}). 
 As well many processes may be modeled by  the logistic differential
equation, or  generalization of it, and the  applications are wide  and in different field of applications  \cite{thieme}.

Sharing the applications with the logistic case, the fractional equation
is a model for the growth of a given population, describing the population
behavior and showing an increase, a saturation and a flat asymptotic
behavior. The global shape is also respected by the fractional logistic
case by numerical evidence;
however, some peculiar differences show that the fractional model is a
good candidate to model  a memory effect on the population (see \eqref {integrodif0}) and that the
fractional order  may be modified  along the
process in order to constrain the growth (see for example \cite{nostro1} for the estimation of fractional order by observations).

The problem to give a solution of the fractional logistic
equation was unsolved and several attempts have been done (see for
instance \cite{attempt1, attempt2, according1, West}).
A close answer  from an empirical point of view is contained in \cite{OrtBen}. The solution we propose here
has  an exact representation given by a closed formula for the coefficients in the convergence set.  
The convergence analysis is stated in Theorem \ref{main}

A model considering a modified logistic equation
has been also treated in \cite{nostro2}.

To our knowledge, from a mathematical point of view,  the fractional
logistic function is undefined.
This paper answers to the question for limited time.
The observation is that for any positive $t$  less than $\pi$, denoted by
$u$ the logistic function
(with all the parameters simplified) the equality holds true
\begin{align}
u(t) = \sum_{k\geq 0}  \frac{E_k}{k!}{t^{ k}},
\label{first-u}
\end{align}
with $E_k$ the Euler's numbers.
In the next subsection we explore this point.

\subsection{The logistic fractional function}
In this subsection we introduce the logistic fractional function.
 As usual by  $\Gamma$ we denote
   $$ \Gamma  (x)=\int_0^{+\infty}t^{x-1}e^{-t} dt  $$
for any real positive $x$.

Let $R>0$ fixed. In $(0,R)$, we introduce the function
\begin{align}
w(t) = \sum_{k\geq 0}  \frac{E^\beta_k}{M^k} \frac{t^{\beta
k}}{\Gamma(\beta k +1)}, \quad M\geq 1, \quad \beta \in (0,1), \quad t <R
\label{serie-goal}
\end{align}
where
\begin{align}
E^\beta_{k+1} = & - \sum_{\begin{subarray}{c} i,j\\ i+j =k
\end{subarray}}\frac{\Gamma( \beta k +1)}{\Gamma(\beta i +1)\Gamma(\beta j
+1)} E^\beta_i E^\beta_j\notag \\
= & - \frac{1}{\beta k +1} \sum_{\begin{subarray}{c} i,j\\ i+j =k
\end{subarray}} \frac{E^\beta_i E^\beta_j }{B(\beta i +1, \beta j +1)},
\quad k \in 2\mathbb{N} \label{eulerBetaN}
\end{align}
with the logistic constrains
\begin{equation}
\begin{split}
E^\beta_1 & =   E^\beta_0 - (E^\beta_0)^2,\\
E^\beta_0 & =  1/2.
\end{split}
\label{log-constr}
\end{equation}

The result we obtain is that,
in the convergence set,  the function $w$ is solution (in the Caputo sense) of the fractional logistic
equation.

 We recall that the Caputo fractional derivative is given by
\begin{align*}
 {^*D}^\beta_t w(t):= \frac{1}{\Gamma(1-\beta)} \int_0^t
\frac{w^\prime(s)}{(t-s)^\beta}ds, \quad t>0, \quad \beta \in (0,1)
\end{align*}
with $w^\prime(s)=dw/ds$.

 Thus, the  fractional logistic equation is
\begin{align}
{^*}D^\beta_t w = w-w^2.
\label{first-frac-log}
\end{align}

 The proof of this result leads to introduce a new class of numbers that
corresponds to Euler's numbers for special value of the fractional order.
This new class that we called generalized  Euler's numbers has an
interest by itself. As we will see in the paper they share with the
Euler's numbers properties helpful to find solutions.
 The method we propose here is suggested by the logistic case and goes
back to the exact solution to the logistic equation in $(0, \pi)$.
However, in the fractional  logistic case other considerations are no
more valid  to solve the problem in the whole positive line.

The fractional equation \eqref{first-frac-log} has the following intersting integro-differential counterpart
\begin{align}
\frac{dw}{d t} = \frac{1}{M} \frac{1}{\Gamma(\beta)} \int_0^t (w(s) -w^2(s)) (t-s)^{\beta -1} ds
\label{integrodif0}
\end{align}
that we also consider further on.

\section{ Logistic equation and Fractional Derivatives}
\subsection{The Euler's numbers}
We first recall the Euler's polynomials
\begin{align*}
E_k(x) = \frac{1}{k+1} \sum_{s=0}^{k} \binom{k+1}{s} 2^s\, B_s(x/2)
\end{align*}
where the Bernoulli polynomials can be written as
\begin{align*}
B_s(x) = \sum_{j=0}^s \binom{s}{j} b_{s-j}\, x^j
\end{align*}
in terms of the Bernoulli numbers $b_j=B_j(0)$ (usually denoted by $B_j^{-}$). In particular, 
\begin{align*}
b_0=1,\quad b_1= -\frac{1}{2}, \quad b_2=\frac{1}{6}, \quad b_3=0, \quad b_4=- \frac{1}{30}, \quad b_5=0, \quad b_6= \frac{1}{42}, \quad b_7=0, \quad b_8=- \frac{1}{30}, \ldots
\end{align*}
We recall the useful formula
\begin{align*}
E_k(x) = \sum_{k=0}^k \binom{k}{s} \frac{E_k}{2^k} \left(x - \frac{1}{2} \right)^k
\end{align*}

\subsection{The logistic equation}
For the sake of simplicity we now assume that $M=1$. Let us consider the function
\begin{align}
u(t) = \frac{e^t}{1+e^t} = \frac{1}{2} \sum_{k\geq 0} E_k(1) \frac{t^k}{k!}, \quad t < \pi
\label{serie-u}
\end{align}
where $E_k(x)$ are the Euler's polynomials defined above. For $k$ odd,  we have that (see \cite[Corollary 3]{dilcher})
\begin{align*}
(-1)^{\frac{k+1}{2}} \frac{\pi^{k+1}}{4 \, \Gamma(k+1)}E_k(x) \to \cos (\pi x)\quad \textrm{ as } k \to \infty.
\end{align*}
We immediately recover the uniform convergence for the series \eqref{serie-u} on compact subsets of $\mathbb{C}$. We observe that, for $\beta =1$,
\begin{align*}
E^1_{k} = \frac{E_k}{2} = \frac{E_k(1)}{2}
\end{align*}
where $E_k=E_k(1)$ are the Euler's polynomials evaluated at $x=1$.

\begin{remark}
Observe that, 
\begin{align*}
U(t) = \sum_{k \geq 0} \left(\frac{a}{g(t)} \right)^k = \frac{g(t)}{a + g(t)} = \left( 1+ e^{-\int_0^t v(s) ds} \right)^{-1}, \quad g(t)<a
\end{align*}
where $v(s)= (\ln g(s))^\prime\, \mathbf{1}_{(g(s)>a)}$ solves the logistic equation
\begin{align*}
U^\prime =  (\ln g)^\prime\, \left( U - U^2 \right), \quad g(t)<a.
\end{align*}
\end{remark}

For the sake of simplicity we equals to one the growth rate and the
carrying capacity considered in the literature. Thus, $(1-u)$ is the
biological potential describing the density dependence, if negative, the
population decreases back to the carrying capacity.
The logistic model may be generalized as follow

\begin{align*}
K^\Phi * u^\prime = f(u)
\end{align*}
where
\begin{align*}
(K^\Phi * u^\prime)(t) := \int_0^t K^\Phi(t-s)\, u^\prime (s)\, ds.
\end{align*}
and
\begin{equation}
f(u)=u(1-u).
\end{equation}

\begin{defin}
A function $\phi : (0, \infty) \mapsto \mathbb{R}$ is a Bernstein function if $\phi$ is of class $C^\infty$, $\phi(\lambda) \geq 0$ for all $\lambda > 0$ and 
\begin{align*}
-(-1)^n  \phi^{(n)}(\lambda) \geq 0 \quad \textrm{for all } n \in \mathbb{N} \textrm{ and } \lambda>0.
\end{align*} 
\end{defin}

We introduce the Bernstein function 
\begin{equation}
\Phi(\lambda) = \int_0^\infty \left( 1 - e^{ - \lambda z} \right) \Pi(dz) 
 \label{LevKinFormula}
\end{equation}
with
\begin{equation}
\frac{\Phi(\lambda)}{\lambda} = \int_0^\infty e^{-\lambda z} \Pi((z, \infty)) dz 
\label{tailSymb}
\end{equation} 
where $\Pi$ on $(0, \infty)$ with $\int (1 \wedge z) \Pi(dz) < \infty$ is the associated L\'evy measure (and $\Pi((z, \infty))$ is the tail of the L\'evy measure $\Pi$). If $\Phi(\lambda) = \lambda^\beta$ for instance, we have that
\begin{align*}
\Pi(dz) = \frac{\beta}{\Gamma(1-\beta)} z^{-\beta- 1} dz \quad \textrm{and} \quad \Pi((z, \infty)) = \frac{z^{-\beta}}{\Gamma(1-\beta)}.
\end{align*}

Let $M>0$ and $w \geq 0$. Let $\mathcal{M}_\omega$ be the set of (piecewise) continuous function on $[0, \infty)$ of exponential order $\omega$ such that $|u(t)| \leq M e^{\omega t}$. Denote by $\widetilde{u}$ the Laplace transform of $u$. Then, we define the operator $\mathfrak{D}^\Phi_t : \mathcal{M}_\omega \mapsto \mathcal{M}_\omega$ such that
\begin{align*}
\int_0^\infty e^{-\lambda t} \mathfrak{D}^\Phi_t u(t)\, dt = \Phi(\lambda) \widetilde{u}(\lambda) - \frac{\Phi(\lambda)}{\lambda} u(0), \quad \lambda > \omega
\end{align*}
where $\Phi$ is given in \eqref{LevKinFormula}. Since $u$ is exponentially bounded, the integral $\widetilde{u}$ is absolutely convergent for $\lambda>\omega$.  By Lerch's theorem the inverse Laplace transforms $u$ and $\mathfrak{D}^\Phi_tu$ are uniquely defined. Since
\begin{align}
\label{PhiConv}
\Phi(\lambda) \widetilde{u}(\lambda) - \frac{\Phi(\lambda)}{\lambda} u(0) = & \left( \lambda \widetilde{u}(\lambda) - u(0) \right) \frac{\Phi(\lambda)}{\lambda}
\end{align}
the operator $\mathfrak{D}^\Phi_t$ can be written as a convolution involving the ordinary derivative and the inverse transform of \eqref{tailSymb} iff $u \in \mathcal{M}_\omega \cap C([0, \infty), \mathbb{R}_+)$ and $u^\prime \in \mathcal{M}_\omega$. Throughout we consider
\begin{align}
\label{general-log}
\mathfrak{D}^\Phi_t u = f(u).
\end{align}
For different definitions and representatuions of the operator $\mathfrak{D}^\Phi_t$ the interested reader can also see the recent works \cite{chen, dovcap, toaldo} and \cite{KST06, pod99} for fractional calculus and derivatives.

\begin{remark}
We observe that for $\Phi(\lambda)=\lambda$ (that is we deal with the ordinary derivative) we have that the equation \eqref{general-log} becomes the logistic equation.
\end{remark}

\begin{remark}
If $\Phi(\lambda)=\lambda^\beta$, the operator $\mathfrak{D}^\Phi_t$ becomes the Caputo fractional derivative
\begin{align*}
\mathfrak{D}^\Phi_t u(t) = {^*D}^\beta_t u(t):= \frac{1}{\Gamma(1-\beta)} \int_0^t \frac{u^\prime(s)}{(t-s)^\beta}ds, \quad t>0, \quad \beta \in (0,1)
\end{align*}
with $u^\prime(s)=du/ds$.  Thus, the equation \eqref{general-log} becomes the equation
\begin{align}
{^*}D^\beta_t u = f(u)
\end{align}
which is well-know in the literature as the fractional logistic equation.
\end{remark}

\begin{remark}
A further example is given by the symbol $\Phi(\lambda)=\lambda^{2\beta} + \lambda^\beta$ for $\beta \in (0, 1/2)$, that is, $\mathfrak{D}^\Phi_t$ becomes the telegraph fractional operator. Since we have infinite Bernstein function, then we can define a number of fractional operators together with the correspond fractional equations \eqref{general-log}. 
\end{remark}
 
\noindent We denote by $D^\beta_t$ the Riemann-Liouville derivative
\begin{align*}
D^\beta_t u(t) := \frac{1}{\Gamma(1-\beta)} \frac{d}{dt} \int_0^t \frac{u(s)}{(t-s)^\beta}ds,  \quad t>0, \quad \beta \in (0,1)
\end{align*}
with Laplace transform 
$$\int_0^\infty e^{-\lambda t} D^\beta_t u(t)\, dt = \lambda^\beta \widetilde{u}(\lambda).$$ 
We recall that 
\begin{align*}
\int_0^\infty e^{-\lambda t} \, {^*D}^\beta_t u(t)\, dt= \lambda^\beta \widetilde{u}(\lambda) - \lambda^{\beta -1}u(0).
\end{align*}

\noindent We also recall the Stirlin's approximation for the Gamma function
\begin{align*}
\Gamma(x) := \int_0^\infty e^{-z} z^{x-1} dz \sim  \sqrt{\frac{2\pi}{x}} \left( \frac{x}{e}\right)^x 
\end{align*}
and the Beta function  
\begin{align}
B(x,y) := \int_0^1 z^{x-1}(1-z)^{y-1}dz \sim \sqrt{2\pi} \frac{x^{x-1/2} y^{y- 1/2}}{(x+y)^{x+y-1/2}}
\label{simBeta}
\end{align}
for large values of $x>0$ and $y>0$. Morever, the following bounds hold true (\cite{boundBetaRef})
\begin{align}
B(x,y) \leq \frac{1}{xy}, \qquad x,y>1
\label{boundBeta}
\end{align}
and (\cite{boundGamma0Ref})
\begin{align}
2^{x -1} \leq \Gamma(x +1) \leq 1, \quad 0 \leq x \leq 1.
\label{boundGamma0}
\end{align}
We also consider the following bounds for the Gamma function (\cite{boundGammaRef})
\begin{align}
\label{boundGamma}
\frac{x^{x-\gamma}}{e^{x-1}}< \Gamma(x) < \frac{x^{x-1/2}}{e^{x-1}}, \quad x>1
\end{align}
where 
\begin{align}
\gamma \approx 0.577215.
\label{E-M-const}
\end{align}
is the Euler-Mascheroni constant. Note that, from \eqref{simBeta},
\begin{align*}
\Gamma(1-\beta)\, | {^*D}^\beta_t u | \leq  \int_0^t | u^\prime(s) | (t-s)^{-\beta} ds \leq  B(1 + \theta, 1-\beta) \quad \textrm{if} \quad | u^\prime(t)| \leq t^\theta, \; \theta>-1
\end{align*}
and, from \eqref{PhiConv}, by Young's inequality,
\begin{align}
\label{YoungSymb}
\int_0^\infty |\mathfrak{D}^\Phi_t u |^p dt \leq \left( \int_0^\infty |u^\prime |^p dt \right) \left( \lim_{\lambda \downarrow 0} \frac{\Phi(\lambda)}{\lambda} \right)^p, \qquad p \in [1, \infty)
\end{align}
where $\lim_{\lambda \downarrow 0} \Phi(\lambda) /\lambda$ is finite only in some cases.

\section{Fractional logistic equation: Main Result and Proof}

\subsection{The fractional equation}
Let $R>0$. We consider the function
\begin{align}
w(t) = \sum_{k\geq 0}  \frac{E^\beta_k}{M^k} \frac{t^{\beta k}}{\Gamma(\beta k +1)}, \quad M\geq 1, \quad \beta \in (0,1), \quad t <R
\label{serie-goal}
\end{align}
where 
\begin{align}
E^\beta_{k+1} = & - \sum_{\begin{subarray}{c} i,j\\ i+j =k \end{subarray}}\frac{\Gamma( \beta k +1)}{\Gamma(\beta i +1)\Gamma(\beta j +1)} E^\beta_i E^\beta_j\notag \\
= & - \frac{1}{\beta k +1} \sum_{\begin{subarray}{c} i,j\\ i+j =k \end{subarray}} \frac{E^\beta_i E^\beta_j }{B(\beta i +1, \beta j +1)}, \quad k \in 2\mathbb{N} \label{eulerBetaN}
\end{align}
with the logistic constrains
\begin{equation}
\begin{split}
E^\beta_1 & =   E^\beta_0 - (E^\beta_0)^2,\\
E^\beta_0 & =  1/2.
\end{split} 
\label{log-constr}
\end{equation}
We refer to the sequence of Euler's fractional numbers $\{E^\beta_n\}_{n \in \mathbb{N}_0}$ as the Euler's $\beta$-numbers (we introduce the Euler's $\Phi$-number further on).

We are not able to prove absolute or uniform convergence of \eqref{serie-goal} in $(0, R)$. Our approach is therefore focused on the Weierstrass M-test in order to obtain different intervals of convergence depending on the fractional order $\beta$ for different values of $\beta$. Thus, we obtain that \eqref{serie-goal} is convergent in some subsets of $(0, R)$, that is $w(t)$ is absolutely convergent $\forall\, t \in (0, r_\beta)$. In particular, we have uniform convergence in $(0, r_\beta)$. Each interval of convergence turns out to be given in terms of $\beta \in (0,1]$ and it depends continuously on $\beta$. As $\beta$ increases to one we recover the solution to the logistic equation and the convergence in $(0, \pi)$ of $u^\prime$ and $u$ in \eqref{first-u} has been extensively investigated in the literature. Concerning the differentiation of $w$ for $\beta \in (0,1)$ we easily have convergence on $(0, r_\beta)$ from \eqref{first-frac-log} and the uniform convergence of the product of uniformly convergent series.

\begin{tm}
\label{main}
Let $\beta \in (0,1)$, $M>1$, $0< r_\beta < R$. Let $\gamma$ be the Euler-Mascheroni constant. 

The series \eqref{serie-goal} uniformly converges in each compact subsets of $(0, r_\beta)$. In particular, 
\begin{itemize}
\item[i)] for $\beta>0$, the radius is given by
\begin{align*}
r_\beta = M \frac{2^\beta}{e^\beta} \left( \frac{3\beta+1}{2\beta +1}\right)^{2\beta + 1/2} (3\beta +1)^{\beta + (1/2-\gamma)} \leq M  c_1 
\end{align*}
where $c_1 \approx 3.074366$ and for $\beta \uparrow1$, $r_\beta \uparrow M c_1$.
\item[ii)] for $\beta > \gamma - 1/2$, the radius is given by
\begin{align*}
r_\beta = M \frac{2}{e^{2\beta} (\beta +1)^{1/2}} \left( \frac{3\beta+1}{2\beta+1} \right)^{2\beta +1/2} \left( 3\beta +1 \right)^{-\beta + \gamma - 1/2} \leq   M c_2
\end{align*}
where $c_2 \approx 3.116623$ and for $\beta \uparrow 1$, $r_\beta \uparrow M c_2$.
\end{itemize}

%
%\begin{align*}
%0 < \rho \leq \left( 3 M^4 \frac{\Gamma(\beta)\Gamma(\beta +1)}{ B(2\beta, \beta)}\right)^{1/4\beta}.
%\end{align*}

The series \eqref{serie-goal} solves the fractional logistic equation
\begin{align}
\label{eq-log-mainTHM}
{^*D}^\beta_t w(t) = \frac{1}{M} \big( w(t) - w^2(t) \big),  \quad w(0)=1/2.
\end{align}
\end{tm}

\begin{proof}
We first show that \eqref{serie-goal} solves \eqref{eq-log-mainTHM} pointwise and then we study the convergence of \eqref{serie-goal}.\\
{\bf 1)} The Riemann-Liouville derivative of $w$ gives the series
\begin{align*}
D^\beta_t w(t) = \sum_{k\geq 0} \frac{E^\beta_k}{M^k} \frac{t^{\beta k- \beta}}{\Gamma(\beta k +1- \beta)}= E^\beta_0 \frac{t^{- \beta}}{\Gamma(1- \beta)} +  \sum_{k\geq 1} \frac{E^\beta_k}{M^k} \frac{t^{\beta k- \beta}}{\Gamma(\beta k +1- \beta)}
\end{align*}
that is, the Caputo derivative of $w$ is given by
\begin{align*}
{^*D}^\beta_t w(t) = D^\beta_t \left(w(t) - w(0) \right) =  \sum_{k\geq 0} \frac{E^\beta_{k+1}}{M^{k+1}} \frac{t^{\beta k}}{\Gamma(\beta k +1)}.
\end{align*}
Let $Lw = w-w^2$ be the logistic operator. From \eqref{serie-goal} we have that
\begin{align*}
L w(t) = & \sum_{k\geq 0} \frac{E^\beta_k}{M^k} \frac{t^{\beta k}}{\Gamma(\beta k +1)} - \sum_{k \geq 0}\sum_{s \geq 0} \frac{E^\beta_k}{M^k} \frac{t^{\beta k}}{\Gamma(\beta k +1)} \frac{E^\beta_s}{M^s} \frac{t^{\beta s}}{\Gamma(\beta s +1)}\\
= & E^\beta_0 - E^\beta_0 E^\beta_0\\
& + \left( \frac{E^\beta_1 }{\Gamma(\beta +1)} - 2\frac{E^\beta_0 E^\beta_1}{\Gamma(\beta +1)} \right) \frac{t^\beta}{M}\\
& + \left( \frac{E^\beta_2}{\Gamma(2\beta +1)} - 2 \frac{E^\beta_0 E^\beta_2}{\Gamma(2\beta +1)} - \frac{E^\beta_1 E^\beta_1}{\Gamma(\beta +1)\Gamma(\beta +1)} \right) \frac{t^{2\beta}}{M^2}\\
& + \left( \frac{E^\beta_3}{\Gamma(3\beta +1)} - 2 \frac{E^\beta_0 E^\beta_3}{\Gamma(3\beta +1)} - 2 \frac{E^\beta_1 E^\beta_2}{\Gamma(\beta +1)\Gamma(2\beta +1)} \right) \frac{t^{3\beta}}{M^3}\\
& + \left(\frac{E^\beta_4}{\Gamma(4\beta +1)} - 2 \frac{E^\beta_0 E^\beta_4}{\Gamma(4\beta +1)} -2\frac{E^\beta_1 E^\beta_3}{\Gamma(\beta+1)\Gamma(3\beta +1)} - \frac{E^\beta_2 E^\beta_2}{\Gamma(2\beta +1) \Gamma(2\beta +1)} \right) \frac{t^{4\beta}}{M^4}\\
& + \ldots .
\end{align*}
Under the constrains \eqref{log-constr} we obtain the coefficients $E^\beta_k$ for $k>1$. In particular, we have that
\begin{align*}
\frac{1}{M}\frac{E^\beta_2}{\Gamma(\beta +1)} = & \frac{E^\beta_1}{\Gamma(\beta +1)} - 2 \frac{E^\beta_0 E^\beta_1}{\Gamma(\beta +1)} \\
\frac{1}{M}\frac{E^\beta_3}{\Gamma(2\beta +1)} = &  - \frac{E^\beta_1 E^\beta_1}{\Gamma(\beta +1)\Gamma(\beta +1)}\\
\frac{1}{M}\frac{E^\beta_4}{\Gamma(3\beta+1)} = & \frac{E^\beta_3}{\Gamma(3\beta +1)} - 2 \frac{E^\beta_0 E^\beta_3}{\Gamma(3\beta +1)}\\
\frac{1}{M}\frac{E^\beta_5}{\Gamma(4\beta +1)} = & -2\frac{E^\beta_1 E^\beta_3}{\Gamma(\beta+1)\Gamma(3\beta +1)} \\
\frac{1}{M}\frac{E^\beta_6}{\Gamma(5\beta +1)} = & \frac{E^\beta_5}{\Gamma(5\beta +1)} - 2 \frac{E^\beta_0 E^\beta_5}{\Gamma(5\beta +1)}\\
\frac{1}{M}\frac{E^\beta_7}{\Gamma(6\beta +1)} = & - 2 \frac{E^\beta_1 E^\beta_5}{\Gamma(\beta+1)\Gamma(5\beta +1)} - \frac{E^\beta_3 E^\beta_3}{\Gamma(3\beta +1) \Gamma(3\beta +1)}\\
\frac{1}{M}\frac{E^\beta_8}{\Gamma(7\beta +1)} =&  \frac{E^\beta_7}{\Gamma(7\beta +1)} - 2 \frac{E^\beta_0 E^\beta_7}{\Gamma(7\beta +1)} \\
\frac{1}{M}\frac{E^\beta_9}{\Gamma(8\beta +1)} = & -2 \frac{E^\beta_1 E^\beta_7}{\Gamma(\beta+1) \Gamma(7\beta +1)} - 2 \frac{E^\beta_3 E^\beta_5}{\Gamma(3\beta +1) \Gamma(5\beta +1)}\\
\frac{1}{M}\frac{E^\beta_{11}}{\Gamma(10 \beta +1)} = & -2 \frac{E^\beta_1 E^\beta_9}{\Gamma(\beta +1) \Gamma(9 \beta +1))} - 2 \frac{E^\beta_3 E^\beta_7}{\Gamma(3\beta +1) \Gamma (7 \beta +1)} - \frac{E^\beta_5 E^\beta_5}{\Gamma(5\beta +1 ) \Gamma(5 \beta +1)}
\end{align*}
that is
\begin{align*}
\frac{1}{M} E^\beta_{n+1}= & -  \sum_{\begin{subarray}{c} i,j\\ i+j=n \end{subarray}} \frac{\Gamma(n\beta +1)}{\Gamma(i\beta +1)\Gamma(j\beta +1)}E^\beta_i E^\beta_j .
\end{align*}
\noindent As a consequence of the logistic conditions \eqref{log-constr} we have that $E^\beta_{2k}=0$. \\

{\bf 2)} Now we prove uniform convergence of \eqref{serie-u}. We first rewrite the coefficient $E^\beta_k$ as follows
\begin{align*}
E^\beta_1=&1/4 = + \Gamma(\beta +1) \left( \frac{E^\beta_1}{\Gamma(\beta +1)}\right) \\
E^\beta_3 = & - \Gamma(2\beta +1) \left( \frac{E^\beta_1}{\Gamma(\beta +1)} \right)^2\\
E^\beta_5 = & + 2\, \Gamma(4\beta +1) \frac{ \Gamma(2\beta +1)}{\Gamma(3\beta +1) } \left( \frac{E^\beta_1}{\Gamma(\beta +1)} \right)^3\\
%= & 2 \, \Gamma(4\beta +1) \frac{E^\beta _3}{\Gamma(3\beta +1)} \frac{E^\beta_1}{\Gamma(\beta +1)}  \\
E^\beta_7 = & - 2^2 \Gamma(6\beta +1) \left( - \frac{\Gamma(4\beta +1) \Gamma(2\beta +1)}{\Gamma(5\beta +1)\Gamma(3\beta +1)} +  \frac{1}{4} \frac{\Gamma(2\beta+1)\Gamma(2\beta +1)}{\Gamma(3\beta +1)\Gamma(3\beta +1)} \right) \left( \frac{E^\beta_1}{\Gamma(\beta +1)} \right)^4\\
%= & 2\Gamma(6\beta +1) \frac{E^\beta_5}{\Gamma(5\beta +1)}\frac{E^\beta_1}{\Gamma(\beta +1)} - \Gamma(6\beta +1)\left( \frac{E^\beta_3}{\Gamma(3\beta +1)} \right)^2\\
E^\beta_9 = & +2^3 \Gamma(8\beta +1) \left( \frac{\Gamma(6\beta +1)\Gamma(4\beta +1)\Gamma(2\beta +1)}{\Gamma(7\beta +1) \Gamma(5\beta +1) \Gamma(3\beta +1)} + \frac{1}{2}  \frac{\Gamma(4\beta +1)}{\Gamma(5\beta +1)} \left( \frac{\Gamma(2\beta +1)}{\Gamma(3\beta +1)}\right)^2 \right) \left( \frac{E^\beta_1}{\Gamma(\beta +1)} \right)^5  \\
%= & -2 \Gamma(8\beta +1) \frac{E^\beta_7}{\Gamma(7\beta +1)} \frac{E^\beta_1}{\Gamma(\beta +1)} -? \Gamma(8\beta +1) 
\end{align*}
from which we get that, for $n\geq 2$,
\begin{align*}
\frac{E^\beta_{n+1}}{\Gamma((n+1)\beta + 1)} =  \frac{(-2)^{n/2}}{2} \left( \frac{E^\beta_1}{\Gamma(\beta +1)} \right)^{\frac{n}{2}+1}\left( \prod_{j=1}^{n/2} \frac{\Gamma(2j\beta +1)}{\Gamma((2j+1)\beta +1)} + R_n \right).
\end{align*}
Since
\begin{align}
\label{Aj}
A_j:=\frac{\Gamma(2j\beta +1)}{\Gamma((2j+1)\beta +1)} < A_{j-1}
\end{align}
we obtain that
\begin{align*}
R_n \leq \prod_{j=1}^{n/2} \frac{\Gamma(2j\beta +1)}{\Gamma((2j+1)\beta +1)}
\end{align*}
and therefore, we have that
\begin{align*}
\bigg| \frac{E^\beta_{n+1}}{\Gamma((n+1)\beta + 1)}  \bigg| \leq & 2^{n/2} \left( \frac{E^\beta_1}{\Gamma(\beta +1)} \right)^{\frac{n}{2}+1} \, \prod_{j=1}^{n/2} \frac{\Gamma(2j\beta +1)}{\Gamma((2j+1)\beta +1)}.
\end{align*}
From \eqref{Aj}, we also obtain that
\begin{align*}
\prod_{j=1}^{n/2} \frac{\Gamma(2j\beta +1)}{\Gamma((2j+1)\beta +1)} \leq \left( \frac{\Gamma(2\beta +1)}{\Gamma(3\beta +1)} \right)^{n/2}
\end{align*}
and therefore,
\begin{align}
\bigg| \frac{E^\beta_{n+1}}{\Gamma((n+1)\beta + 1)}  \bigg| \leq &  \frac{E^\beta_1}{\Gamma(\beta +1)} \left( 2  \frac{E^\beta_1}{\Gamma(\beta +1)} \frac{\Gamma(2\beta +1)}{\Gamma(3\beta +1)} \right)^{n/2}.
\label{tmp-bound}
\end{align}
Notice that $E^\beta_{2k}=0$. We now continue our analysis from \eqref{tmp-bound} and obtain two related results: \\

\noindent {\it i)} By taking into account \eqref{boundGamma} and the left-side of \eqref{boundGamma0} for $1/\Gamma(\beta +1)$, we get that
\begin{align}
\bigg| \frac{E^\beta_{n+1}}{\Gamma((n+1)\beta + 1)}  \bigg| \leq & \frac{E^\beta_1}{\Gamma(\beta +1)} \left( \frac{1}{2^\beta} \left( \frac{2\beta+1}{3\beta+1}\right)^{2\beta +1/2} \frac{e^\beta}{(3\beta+1)^{\beta+1/2-\gamma}} \right)^{n/2} =a_{n}.
\end{align}
The series
\begin{align*}
\sum_{n \geq 0} (- 1)^n a_{2n}\, t^n 
\end{align*}
uniformly converges in $(0, r_\beta)$ with $\beta \in (0,1)$ and
\begin{align*}
 r_\beta = \frac{2^\beta}{e^\beta} \left( \frac{3\beta+1}{2\beta +1}\right)^{2\beta + 1/2} (3\beta +1)^{\beta + (1/2-\gamma)} \leq 3.074366 .
\end{align*}

\noindent {\it ii)} 
By taking into account \eqref{boundGamma}, we get
\begin{align*}
\bigg| \frac{E^\beta_{n+1}}{\Gamma((n+1)\beta + 1)}  \bigg| \leq & \frac{E^\beta_1}{\Gamma(\beta +1)} \left( \frac{1}{2} \left(\frac{2\beta+1}{3\beta +1} \right)^{2\beta +1/2} \frac{e^{2\beta}}{(3\beta +1)^{\beta+1/2-\gamma}} \frac{1}{(\beta+1)^{\beta +1 -\gamma}} \right)^{n/2} =b_{n}.
\end{align*} 
The series
\begin{align*}
\sum_{k\geq 0} (-1)^n b_{2n} \, t^n
\end{align*}
uniformly converges in $(0, r_\beta)$ with $\beta>\gamma - 1/2$ and
\begin{align*}
r_\beta = \frac{2}{e^{2\beta} (\beta +1)^{1/2}} \left( \frac{3\beta+1}{2\beta+1} \right)^{2\beta +1/2} \left( 3\beta +1 \right)^{-\beta + \gamma - 1/2} \leq   3.116623.
\end{align*}
\end{proof}

\begin{remark}
By taking into account \eqref{boundBeta}, from \eqref{tmp-bound} we get that
\begin{align*}
\bigg| \frac{E^\beta_{n+1}}{\Gamma((n+1)\beta + 1)}  \bigg| \leq & \frac{E^\beta_1}{\Gamma(\beta +1)} \left( 2 \frac{E^\beta_1}{\Gamma(\beta +1)} \frac{(3\beta +1)}{\Gamma(\beta +1)} B(2\beta+1, \beta+1) \right)^{n/2} \\
\leq & \frac{E^\beta_1}{\Gamma(\beta +1)} \left( 2 \frac{E^\beta_1}{\Gamma(\beta +1)} \frac{(3\beta +1)}{\Gamma(\beta +1)} \frac{1}{(\beta +1)(2\beta +1)} \right)^{n/2} \\
\leq & \frac{E^\beta_1}{\Gamma(\beta +1)} \left( \frac{1}{2^{2\beta -1}}  \frac{(3\beta +1)}{(3\beta +1)+ 2\beta^2} \right)^{n/2}
\end{align*}
and the series
\begin{align*}
\frac{1}{4\, \Gamma(\beta +1)} \sum_{k\geq 0} \left( - \frac{1}{2^{2\beta -1}}  \frac{(3\beta +1)}{(3\beta +1)+ 2\beta^2} \, t \right)^k
\end{align*}
is uniformly convergent in $(0, r_\beta)$ with $\beta >1/2$ and 
\begin{align*}
r_\beta = 2^{2\beta -1} \left( 1 + \frac{2\beta^2}{2\beta +1} \right) \leq 3.
\end{align*}
From the last series we obtain that, for $M=1$
\begin{align*}
w(t) \leq &  \frac{1}{2} + \frac{1}{4\, \Gamma(\beta +1)} \left( 1- \frac{2}{2^{2\beta}} \frac{3\beta +1}{3\beta +1 + 2\beta^2} t \right)^{-1}, \quad t < r_\beta.
\end{align*}
Notice that, we easily obtain the analogue result given for $M>1$.
\end{remark}

\subsection{The Euler's $\Phi$-numbers}
The Euler's $\beta$-numbers have been introduced in the previous section as the sequence
\begin{align*}
E^\beta_0=1/2, \quad E^\beta_1= 1/4, \quad E^\beta_{k+1} \textrm{ as in } \eqref{eulerBetaN} \textrm{ for } k \in 2\mathbb{N}
\end{align*}
such that
\begin{align*}
w(t) = \sum_{k \geq 0} E^\beta_k \frac{t^{\beta k}}{\Gamma(\beta k +1)}
\end{align*}
solves the fractional logistic equation
\begin{align*}
{^*D}^\beta_t w = w-w^2.
\end{align*}
The Euler's $\Phi$-numbers can be therefore defined as the sequence $\{E^\Phi_k\}_{k \in \mathbb{N}_0}$ such that
\begin{align*}
v(t) = \sum_{k \geq 0} E^\Phi_k \frac{t^{\beta k}}{\Gamma(\beta k +1)}
\end{align*}
solves the fractional logistic equation
\begin{align*}
\mathfrak{D}^\Phi_t v = v-v^2.
\end{align*}

\subsection{Integro-differential equation}
Let us focus on the Sonine kernels
\begin{align*}
\kappa(t) = \frac{t^{-\beta}}{\Gamma(1- \beta)}, \qquad \overline{\kappa}(t) =  \frac{t^{\beta-1}}{\Gamma(\beta)}
\end{align*}
for which we have that
\begin{align*}
(\kappa * \overline{\kappa} )(t):= \int_0^t \kappa(s) \overline{\kappa}(t-s)ds = \frac{1}{\Gamma(1-\beta) \Gamma(\beta)} \int_0^1 s^{(1- \beta) - 1} (1-s)^{\beta -1} ds = 1.
\end{align*}
By considering the Laplace transform techniques, we immediately have that
\begin{align*}
({^*D}^\alpha_t u * \overline{\kappa})(t) = {^*D}^{\alpha - \beta}_t u
\end{align*}
and
\begin{align*}
\frac{d}{d t} ({^*D}^\beta_t u * \overline{\kappa})(t) = \frac{d u}{dt}.
\end{align*}
The equation \eqref{eq-log-mainTHM} can be therefore rewritten as
\begin{align}
\frac{dw}{d t} = \frac{1}{M} \frac{1}{\Gamma(\beta)} \int_0^t (w(s) -w^2(s)) (t-s)^{\beta -1} ds
\label{integrodif}
\end{align}
where a singular kernel is involved. We underline that the function \eqref{serie-goal} solves the integro-differential equation \eqref{integrodif} and the fractional logistic equations \eqref{eq-log-mainTHM}.

\subsection{Acknowledgement}

We thank Sima Sarv Ahrabi for pointing out to the authors attention the paper \cite{attempt1}.

\end{document}